\documentclass[12pt]{article}
\usepackage{amssymb, latexsym, amsmath, amsthm, moreverb}

\newcommand{\bp}{{\boldsymbol{p}}}
\newcommand{\bq}{{\boldsymbol{q}}}

\newcommand{\bv}{{\boldsymbol{v}}}
\newcommand{\bx}{{\boldsymbol{x}}}
\newcommand{\be}{{\boldsymbol{e}}}
\newcommand{\bone}{{\boldsymbol{1}}}
\newcommand{\cN}{{\cal N}}
\newcommand{\C}{{\mathbb C}}

\newcommand{\N}{{\mathbb N}}

\newcommand{\R}{{\mathbb R}}
\newcommand{\T}{{\mathbb T}}
\newcommand{\Z}{{\mathbb Z}}

\newtheorem{Theorem}{Theorem}[section]

\newtheorem{Corollary}[Theorem]{Corollary}

\newtheorem{Lemma}[Theorem]{Lemma}

\theoremstyle{definition}
\newtheorem{Definition}[Theorem]{Definition}
\newtheorem{Example}[Theorem]{Example}
\newtheorem{Remark}[Theorem]{Remark}

\newcommand{\footnoteremember}[2]{
  \footnote{#2}
  \newcounter{#1}
  \setcounter{#1}{\value{footnote}}
} \newcommand{\footnoterecall}[1]{
  \footnotemark[\value{#1}]
}

\newcommand{\ol}{\overline}
\newcommand{\bil}[2]{\langle{#1},{#2}\rangle}

\newcommand{\Hess}{{\rm Hess}}

\renewcommand{\subset}{\subseteq}

\DeclareTextFontCommand{\textnf}{\normalfont}
\newcommand{\marginnote}[1]{\vrule width0pt height0pt depth0pt
    \vadjust{\vbox to0pt{\vss\hbox to\hsize{\hskip\hsize\quad
    #1\hss}\vskip1.5pt}}}

\begin{document}

\title{A real algebra perspective on multivariate tight wavelet frames}

\author{
Maria Charina\footnoteremember{myfootnote}{Fakult\"at f\"ur Mathematik, TU
Dortmund, D--44221 Dortmund, Germany, $maria.charina@tu-dortmund.de$}, \ Mihai
Putinar \footnote{Department of Mathematics, University of California at Santa
Barbara, CA 93106-3080, USA},\\
 Claus Scheiderer \footnote{Fachbereich
Mathematik und Statistik, Universit\"at Konstanz, D--78457 Konstanz, Germany} \
and \ Joachim St\"ockler \footnoterecall{myfootnote} }

\maketitle

\begin{abstract}
Recent results from real algebraic geometry and the theory
of polynomial optimization are related in a new framework to the
existence question of multivariate tight wavelet frames whose generators have at least
one vanishing moment. Namely, several equivalent formulations of
the so-called Unitary Extension Principle (UEP) from \cite{RS95} are interpreted in terms of
hermitian sums of squares of certain nonnegative trigonometric polynomials and in terms
of semi-definite programming. The latter together with the results in \cite{LP,sch:mz}
answer affirmatively the long standing open
question of the existence of such tight wavelet frames in dimension $d=2$; we also
provide numerically efficient methods for checking their existence and actual construction in
any dimension. We exhibit a class of counterexamples in dimension $d=3$ showing that,
in general, the UEP property is not sufficient for the existence of tight wavelet frames.
On the other hand we provide stronger sufficient conditions for the existence of tight wavelet frames
in dimension $d \ge 3$ and illustrate our results by several examples.
\end{abstract}

%%%%%%%%%%%%%%%%%%%%%%%%%%%%%%%%%%%%%%%%%%%%%%%%%%%


\noindent{\bf Math. Sci. Classification 2000:} 65T60, 12D15, 90C26, 90C22.

\section{Introduction}

Several fundamental results due to two groups of authors (I. Daubechies, B.
Han, A. Ron, Z. Shen \cite{DHRS03} and respectively C. Chui, W. He, J. St\"ockler
\cite{CHS04, CHS05}) lay at the foundation of the theory of tight wavelet frames and also
provide their characterizations. These characterizations allow on one hand
to establish the connection between frame constructions and the challenging
algebraic problem of existence of sums of squares representations (sos) of
non-negative trigonometric polynomials. On the other hand, the same characterizations
provide methods, however unsatisfactory from the practical
point of view, for constructing tight wavelet frames.

The existence and effective
 construction of tight frames, together with good estimates on the number of
 frame generators, are still open problems. One can easily be discouraged
by a general  result by Scheiderer in \cite{S99}, which implies that not all non-negative
trigonometric polynomials in the dimension $d \ge 3$ possess sos
representations. However, the main focus is on dimension $d=2$ and on
special non-negative trigonometric polynomials. This
motivates us to pursue the issue of existence of sos representations further.

It has been observed in \cite{CoifmanDonoho1995} that redundancy of wavelet
frames has advantages for applications in signal denoising - if the data is
redundant, then loosing some data during transmission does not necessarily
affect the reconstruction of the original signal. Shen et al. \cite{Shen2011}
use the tight wavelet frame decomposition to recover a clear image from a single
motion-blurred image. In \cite{JoBra} the authors show how to use
multiresolution wavelet filters $p$ and $q_j$ to construct irreducible
representations for the Cuntz algebra and, conversely, how to recover wavelet
filters from these representations. Wavelet and frame decompositions for
subdivision surfaces are one of the basic tools, e.g.,  for progressive
compression of $3-d$ meshes or interactive surface viewing
\cite{CS07,KSS,VisualComp}. Adaptive numerical methods based on wavelet frame
discretizations have produced very promising results \cite{CDD1,CDD2} when
applied to a large class of operator equations, in particular, PDE's and integral
equations. We list some existing constructions of compactly supported MRA wavelet tight
frames of $L_2(\R^d)$ \cite{CH,CHS01,DHRS03,HanMo2005,LS06,RS95, Selesnick2001}
that employ the Unitary Extension Principle. For any dimension and in the case
of a general expansive dilation matrix, the existence of tight wavelet frames is
always ensured by \cite{CCH,CS}, if the coefficients of the associated
refinement equation are real and nonnegative. A few other compactly
supported multi-wavelet tight frames are circulating nowadays, see \cite{CCH,
CS07,GGL}.

The main goal of this paper is to relate the existence of multivariate tight wavelet frames
to recent advances in
 real algebraic geometry and
the theory
of moment problems. The starting point of our
study is the so-called Unitary Extension Principle (UEP) from \cite{RS95}, a special case of the
above mentioned characterizations in \cite{CHS04, CHS05, DHRS03}. In section
\ref{sec:UEP} we first list several equivalent well-known formulations of the UEP
from wavelet and frame literature, but use the novel algebraic terminology to state them.
It has been already observed in \cite{LS06} that a sufficient condition for the existence
of tight wavelet frames satisfying UEP can be expressed in terms of
sums of squares representations of a certain nonnegative trigonometric polynomial. In
\cite[Theorem 3.4]{LS06}, the authors also provide an algorithm for the actual construction of
the corresponding frame generators. In subsection \ref{subsec:sumsofsquares}, we extend
the result of \cite[Theorem 3.4]{LS06} and obtain another equivalent formulation of UEP,
which combined with the results from \cite{sch:mz} guarantees the existence of UEP tight wavelet
frames in the two-dimensional case, see subsection \ref{subsec:existence}. We also exhibit there a
class of three-dimensional counterexamples showing that,
in general, the UEP conditions are not sufficient for the existence of tight wavelet frames.
In those examples, however, we make a rather strong assumption on the underlying refinable
function, which leaves hope that in certain other cases we will be able to show the
existence of such tight wavelet frames.  The novel,
purely algebraic equivalent formulation of the UEP in Theorem \ref{th:UEP_hermitian} is aimed at
better understanding the structure of tight wavelet frames.
The constructive method in \cite[Theorem 3.4]{LS06} yields frame generators of support
twice as large as the one of the underlying refinable function. Theorem \ref{th:UEP_hermitian} leads to
a numerically efficient method for frame constructions with no such restriction on the size
of their support. Namely, in subsection \ref{subsec:semi-definite}, we show how to reformulate
Theorem \ref{th:UEP_hermitian}  equivalently as a problem of semi-definite programming. This establishes a connection between
constructions of tight wavelet frames and moment problems, see \cite{HP,
lasserrebook, LP} for details.

In section \ref{subsec:sufficient},  we give sufficient
conditions for the existence of tight wavelet frames in dimension $d \ge 3$ and
illustrate our results by several examples of three-dimensional subdivision.
In section \ref{subsec:construction}, we discuss an
elegant method that sometimes simplifies the frame construction and allows to
determine the frame generators analytically. We illustrate this method on the example of
the so-called butterfly scheme from
\cite{GDL}.

{\bf Acknowledgements.} The authors are grateful to the Mathematical Institute at Oberwolfach
for offering optimal working conditions through the Research In Pairs program in 2011. The second author was partially
supported by the National Science Foundation Grant DMS-10-01071.

\section{Background and Notation}

\subsection{Real algebraic geometry}
Let $d\in\N$, let $T$ denote the $d$-dimensional anisotropic real
(algebraic) torus, and let $\R[T]$ denote the (real) affine
coordinate ring of $T$
$$
 \R[T]\>=\>\R\bigl[x_j,\,y_j\colon j=1,\dots,d\bigr]\big/
 \bigl(x_j^2+y_j^2-1\colon j=1,\dots,d\bigr).
$$
In other words, $T$ is the subset of $\R^{2d}$ defined by the
equations $x_j^2+y_j^2=1, 1 \leq j \leq d,$ and endowed with additional algebraic structure
which will become apparent in the following pages.
Rather than working with the above description, we will mostly employ the
complexification of $T$, together with its affine coordinate ring
$\C[T]=\R[T]\otimes_\R\C$. This coordinate ring comes with a natural
involution $*$ on $\C[T]$, induced by complex conjugation. Namely,
$$
 \C[T]\>=\>\C[z_1^{\pm1},\dots,z_d^{\pm1}]
$$
is the ring of complex Laurent polynomials, and $*$ sends $z_j$ to
$z_j^{-1}$ and is complex conjugation on coefficients. The real
coordinate ring $\R[T]$ consists of the $*$-invariant polynomials in
$\C[T]$, i.e. $\displaystyle p=\sum_{\alpha \in \Z^d} p(\alpha)
z^\alpha \in \R[T]$ if and only if $p(-\alpha)=\overline{p(\alpha)}$.

The group of $\C$-points of $T$ is $T(\C)=(\C^*)^d=\C^*\times\cdots \times\C^*$.
In this paper we often denote the group of $\R$-points of $T$ by $\T^d$.
Therefore,
$$
 \T^d\>=\>T(\R)\>=\>\{(z_1,\dots,z_d)\in(\C^*)^d\colon|z_1|=\cdots=
 |z_d|=1\}
$$
is the direct product of $d$ copies of the circle group $S^1$. The
neutral element of this group we denote by $\bone=(1,\dots,1)$.

Via the exponential map $\hbox{exp}$, the coordinate ring
$\C[T]=\C[z_1^{\pm1}, \dots,z_d^{\pm1}]$ of $T$ is identified with the
algebra of (complex) trigonometric polynomials. Namely,
$\hbox{exp}$ identifies $(z_1, \dots, z_d)$ with $\be^{-i \omega}:=(
e^{-i\omega_1}, \dots, e^{-i\omega_d})$. In the same way, the real coordinate ring $\R[T]$
is identified with the ring of real trigonometric polynomials,
i.e.\ polynomials with real coefficients in $\cos(\omega_j)$ and
$\sin(\omega_j)$, $j=1,\dots,d$.
\medskip

Let $M\in\Z^{d\times d}$ be a matrix with $\det(M)\ne0$, and write
$m:=|\det M|$. The finite abelian group
\begin{equation}\label{def:G}
G:=2\pi M^{-T}\Z^d/ 2\pi \Z^d
\end{equation}
is (via exp) a subgroup of $\T^d=T(\R)$. It is isomorphic to
$\Z^d/M^T\Z^d$ and has order $|G|=m$. Its character group is $G'=
\Z^d/ M\Z^d$, via the natural pairing
$$
 G\times G'\to\C^*,\quad\bil\sigma\chi=e^{i\sigma\cdot\chi},
 \quad \sigma\in G, \quad \chi\in G'.
$$
Here $\sigma\cdot\chi$ is the ordinary
inner product on $\R^d$, and $\bil\sigma \chi$ is a root of unity of
order dividing $m$. Note that the group $G$ acts on the coordinate
ring $\C[T]$ via multiplication on the torus
\begin{equation}\label{def:psigma}
   p\mapsto p^\sigma(\be^{-i\omega}):= p(\be^{-i(\omega+\sigma)}), \quad \sigma\in
   G, \quad \omega \in \R^d.
\end{equation}
The group action commutes with the involution~$*$, that is, $(p^*)
^\sigma=(p^\sigma)^*$ holds for $p\in\C[T]$ and $\sigma\in G$.

From the action of the group $G$ we get an associated direct sum
decomposition of $\C[T]$ into the eigenspaces of this action
$$
 \C[T]\>=\>\bigoplus_{\chi\in G'}\C[T]_\chi,
$$
where $\C[T]_\chi$ consists of all $p\in\C[T]$ satisfying $p^\sigma=
\bil\sigma\chi\,p$ for all $\sigma\in G$. For $\chi\in G'$ and $p\in
\C[T]$, we denote by $p_\chi$ the weight $\chi$ isotypical component
of~$p$. Thus,
$$
 p_\chi=\frac1m\sum_{\sigma\in G}\bil\sigma\chi\,p^{\sigma}
$$
lies in $\C[T]_\chi$, and we have $\displaystyle p=\sum_{\chi\in G'}p_\chi$.
For every $\chi\in G'$, we choose a lift $\alpha_\chi\in\Z^d$ such that
\begin{equation}\label{def:p_polyphase}
\tilde p_\chi:=z^{-\alpha_\chi}p_\chi
\end{equation}
 is $G$-invariant. The components $\tilde p_\chi$  are called
 {\em polyphase components} of
 $p$, see
 \cite{StrangNguyen}.
 \bigskip

\subsection{Wavelet tight frames}
A wavelet tight frame is a structured system of functions that has some special
group structure and is defined by the actions of translates and dilates on a
finite set of functions $\psi_j \in L_2(\R^d)$, $ 1 \le j \le N$. More
precisely, let $M\in\Z^{d\times d}$ be a general expansive matrix, i.e.
$\rho(M^{-1})<1$, or equivalently, all eigenvalues of $M$ are strictly larger
than $1$ in modulus, and let $m=|\det M|$.
  We define translation operators $T_\alpha$ on $L_2(\R^d)$
by $T_\alpha f=f(\cdot-\alpha) $, $\alpha\in \Z^d$, and
dilation (homothethy) $U_M$  by $U_M f=m^{1/2} f(M\cdot)$. Note that these
operators are isometries on $L_2(\R^d)$.

\begin{Definition}  \label{def:wavelet_tight_frame}
Let $\{\psi_j \ : \ 1 \le j \le N \}\subset L_2(\R^d)$.  The family
$$
  \Psi=\{U_M^\ell T_\alpha \psi_j \ : \ 1\le j \le N, \
 \ell \in \Z, \ \alpha \in \Z^d \}
$$
is a wavelet tight frame of $L_2(\R^d)$, if
\begin{equation}\label{def:parseval}
   \|f\|^2_{L_2}=\sum_{{1 \le j \le N, \ell \in \Z,} \atop
\alpha \in \Z^d} |\langle f,U_M^\ell T_\alpha \psi_j\rangle|^2 \quad \hbox{for
all} \quad f \in L_2(\R^d).
\end{equation}
\end{Definition}

The foundation for the construction of multiresolution wavelet basis or wavelet
tight frame is a compactly supported function $\phi \in L_2(\R^d)$
with the following properties.

\begin{itemize}
 \item[(i)] $\phi$ is refinable, i.e. there exists a finitely
 supported sequence $p=\left( p(\alpha)\right)_{\alpha \in \Z^d}$, $p(\alpha)\in \C$, such that
 $\phi$ satisfies
 \begin{equation} \label{eq:refinement_equation}
   \phi(x)= m^{1/2} \sum_{\alpha \in \Z^d} p(\alpha) U_M T_\alpha \phi(x),
   \quad x \in \R^d.
 \end{equation}
 Taking the Fourier-Transform
 $$
  \widehat{\phi}(\omega)=\int_{\R^d} \phi(x) e^{-i\omega \cdot x}
  dx
 $$
 of both sides of \eqref{eq:refinement_equation} leads to its equivalent form
 \begin{equation} \label{eq:F_refinement_equation}
  \widehat{\phi}(M^T\omega)=p(\be^{-i\omega})
  \widehat{\phi}(\omega), \quad
  \omega \in \R^d,
 \end{equation}
 where the trigonometric polynomial $p \in \C[T]$ is given by
 $$
  p(\be^{-i\omega})= \sum_{\alpha \in \Z^d} p(\alpha) e^{-i\alpha\omega},
  \qquad \omega\in\R^d.
 $$
The isotypical components $p_\chi$ of $p$ are given by
\begin{equation}\label{def:isotypical}
  p_\chi(\be^{-i\omega})= \sum_{\alpha\equiv \chi~{\rm mod}~M\Z^d} p(\alpha)
  e^{-i\alpha\omega},\qquad \chi\in G'.
\end{equation}

 \item[(ii)] One usually assumes that $\widehat{\phi}(0)=1$ by proper
normalization. This
 assumption on $\widehat{\phi}$ and \eqref{eq:F_refinement_equation} allow us to
 read all properties of $\phi$
 from the polynomial $p$, since   the refinement equation
 \eqref{eq:F_refinement_equation} then implies
 $$
    \widehat{\phi}(\omega)=\prod_{\ell=1}^\infty p(\be^{-i(M^T)^{-\ell} \omega}),
    \quad \omega \in \R^d.
 $$
The uniform convergence of this infinite product on compact sets is guaranteed
by $p(\bone)=1$.

\item[(iii)] One of the approximation properties of  $\phi$ is the requirement
that the translates $T_\alpha \phi$, $\alpha\in\Z^d$, form a partition of unity.
Then
\begin{equation}\label{identity:isotypical_at_one}
  p_\chi(\bone)=m^{-1} ,\qquad \chi\in G'.
\end{equation}
\end{itemize}

\noindent  The functions $\psi_j$, $j=1, \dots, N$, are assumed to
be of the form
\begin{equation}\label{def:psij}
 \widehat{\psi}_j(M^T\omega)=q_j(\be^{-i\omega}) \widehat{\phi}(\omega),
\end{equation}
where $q_j \in \C[T]$. These assumptions
imply that $\psi_{j}$ have compact support and, as in
\eqref{eq:refinement_equation}, are finite linear combinations of
$U_M T_\alpha\phi$.

%-------------------------------------------------------------------%
\section{Equivalent formulations of UEP} \label{sec:UEP}

In this section we first recall the method called UEP (unitary extension
principle) that allows us to determine the trigonometric
polynomials $q_j$, $1\le j \le N$,
such that the family $\Psi$  in Definition \ref{def:wavelet_tight_frame} is a
wavelet tight frame of $L_2(\R^d)$, see \cite{DHRS03,RS95}.
 We also give several
equivalent formulations of UEP to link frame constructions with problems in
algebraic geometry and semi-definite programming.

We assume throughout this section
that $\phi\in L_2(\R^d)$
is a refinable function with respect to the expansive matrix
$M\in \Z^{d\times d}$, with trigonometric polynomial $p$ in
\eqref{eq:F_refinement_equation} and $\hat \phi(0)=1$, and
 the functions $\psi_j$ are defined as in \eqref{def:psij}.
We also make use of the definitions \eqref{def:G} for $G$
and \eqref{def:psigma} for $p^\sigma$, $\sigma\in G$.

\subsection{Formulations of UEP in wavelet frame literature}

Most formulations of the UEP are given in terms of identities for trigonometric
polynomials, see \cite{DHRS03,RS95}.

\begin{Theorem} \label{th:UEP} (UEP) Let the
trigonometric polynomial $p \in \C[T]$ satisfy $p(\bone)=1$. If the
trigonometric polynomials $q_j \in \C[T]$, $1 \le j \le N$, satisfy the
identities
\begin{equation}\label{id:UEP}
 \delta_{\sigma,\tau}-p^{\sigma*}p^{\tau}=
 \sum_{j=1}^N q_j^{\sigma*}q_j^{\tau},\qquad
 \sigma,\tau \in G,
\end{equation}
then the family  $\Psi$ is a wavelet tight frame of $L_2(\R^d)$.
\end{Theorem}

We next state another equivalent formulation of the Unitary Extension Principle in Theorem
\ref{th:UEP} in terms of the isotypical components $p_\chi$,
$q_{j,\chi}$ of the polynomials $p$, $q_j$. In the wavelet and frame literature, see e.g.
\cite{StrangNguyen}, this equivalent formulation of UEP is usually given in terms of the
polyphase components in \eqref{def:p_polyphase} of $p$ and $q_j$. The proof we
present here serves as an illustration of the algebraic structure behind wavelet
and tight wavelet frame constructions.

\begin{Theorem} \label{th:UEP_polyphase}
Let the trigonometric polynomial $p \in \C[T]$ satisfy $p(\bone)=1$. The
identities \eqref{id:UEP} are equivalent to
\begin{equation}\label{id:equiv_UEP_1}
\begin{array}{rcl}
 && \displaystyle p_{\chi}^* p_\chi+\sum_{j=1}^N q_{j,\chi}^* q_{j,\chi}=m^{-1},
 \quad \chi \in G', \\[12pt]
 && \displaystyle p_{\chi}^* p_{\eta}+\sum_{j=1}^N q_{j,\chi}^* q_{j,\eta}=0\qquad
 \chi,\eta \in G', \quad \chi \not=\eta.
 \end{array}
\end{equation}
\end{Theorem}

\begin{proof}
 Recall that
 $ \displaystyle{
   p=\sum_{\chi \in G'} p_{\chi}}$  and  $ \displaystyle{
   p_\chi=m^{-1} \sum_{\sigma \in G}\langle\sigma, \chi \rangle p^{\sigma}}$ imply
$$
   p^*=\sum_{\chi \in G'} p^*_{\chi} \quad \hbox{and} \quad  p^{\sigma*}=\sum_{\chi \in G'}
   (p^*_{\chi})^\sigma = \sum_{\chi \in G'} \langle\sigma, \chi\rangle
   p^*_{\chi}.
$$
 Thus, with $\eta'=\chi+\eta$ in the next identity, we get
$$
 p^{\sigma *} p = \sum_{\chi, \eta' \in G'} \langle \sigma, \chi\rangle p^*_\chi
 p_{\eta'}=\sum_{\eta \in G'}  \sum_{\chi \in G'}
 \langle\sigma, \chi\rangle p^*_\chi p_{\chi+\eta}.
$$
Note that  the isotypical components of $p^{\sigma *} p$ are given by
\begin{equation}\label{id:thmUEPpolyphase}
 (p^{\sigma *} p)_\eta =   \sum_{\chi \in G'}
 \langle\sigma, \chi\rangle p^*_\chi p_{\chi+\eta},\qquad \eta\in G'.
\end{equation}
Similarly for $q_j$. Therefore, we get that the identities \eqref{id:UEP} for $\tau=0$  are
equivalent to
\begin{equation} \label{id:equiv_UEP_1_aux}
  % && \sum_{\chi \in G'} \left( P_\chi P^*_{\chi}+ \sum_{j=1}^N Q_{j,\chi} Q^*_{j,\chi}
  % \right)=1, \quad \eta=0, \quad  \sigma=0, \\
  % && \sum_{\chi \in G'} <\sigma, \chi> \left( P_\chi P^*_{\chi}+ \sum_{j=1}^N Q_{j,\chi} Q^*_{j,\chi}
  % \right)=0, \quad \eta=0, \quad  \sigma \in G'\setminus \{0\}, \notag \\
    \sum_{\chi \in G'} \langle\sigma, \chi\rangle \left( p^*_\chi p_{\chi+\eta}+ \sum_{j=1}^N q^*_{j,\chi} q_{j,\chi+\eta}
   \right)=\delta_{\sigma,0} \delta_{\eta,0}, \quad \eta \in G', \quad  \sigma \in G.
\end{equation}
Note that the identities  \eqref{id:UEP} for $\tau \in G$ are redundant and it
suffices to consider only those for $\tau=0$. For fixed $\eta \in G'$, \eqref{id:equiv_UEP_1_aux} is a
system of $m$ equations indexed by $\sigma \in G$ in $m$ unknowns $\displaystyle
p_\chi^* p_{\chi+\eta}+ \sum_{j=1}^N q^*_{j,\chi} q_{j,\chi+\eta}$, $\chi \in
G'$. The corresponding system matrix $A=(\langle \sigma, \chi \rangle)_{\sigma
\in G, \chi \in G'}$ is invertible and $\displaystyle A^{-1}=m^{-1} A^*$. Thus,
\eqref{id:equiv_UEP_1_aux} is equivalent to \eqref{id:equiv_UEP_1}.
 \end{proof}

It is easy to see that  the identities in Theorem \ref{th:UEP} and in Theorem
\ref{th:UEP_polyphase} have equivalent matrix formulations.

\begin{Theorem}\label{th:matrixUEP}
The identities \eqref{id:UEP} are equivalent to
\begin{equation}\label{identity:UEPmatrixform}
  U^*U=I_{m}
\end{equation}
with
$$
   U^*=\left[
  \begin{matrix} p^{\sigma*}& q_1^{\sigma*} &\cdots& q_N^{\sigma*}\end{matrix}
  \right]_{\sigma\in G} \in M_{m \times (N+1)}(\C[T]),
$$
and are also equivalent to
\begin{equation}\label{identity:UEPmatrixformpoly}
  \widetilde{U}^*\widetilde{U}=m^{-1}I_{m},
\end{equation}
with
$$
   \widetilde{U}^*=\left[
  \begin{matrix} \tilde{p}^*_{\chi}&
  \tilde{q}_{1,\chi}^{*} &\cdots& \tilde{q}_{N,\chi}^{*}\end{matrix}
  \right]_{\chi \in G'} \in M_{m \times (N+1)}(\C[T]).
$$
\end{Theorem}

\begin{Remark}
The identities \eqref{identity:UEPmatrixform} and
\eqref{identity:UEPmatrixformpoly} connect the construction of $q_1,\ldots,q_N$
to the following matrix extension problem: extend the first row
$(p^\sigma)_{\sigma\in G}$ of the polynomial matrix $U$ (or $(\widetilde
p_\chi)_{\chi\in G'}$ of $\widetilde U$) to a rectangular $(N+1)\times m$
polynomial matrix satisfying \eqref{identity:UEPmatrixform} (or
\eqref{identity:UEPmatrixformpoly}). There are two major differences
between the identities
\eqref{identity:UEPmatrixform} and \eqref{identity:UEPmatrixformpoly}.
While the first column $(p,q_1,\ldots,q_N)$ of  $U$ determines all
other columns of $U$ as well, the columns of the matrix $\widetilde U$
can be
chosen independently, see \cite{StrangNguyen}. All entries of $\widetilde U$,
 however, are forced to be $G$-invariant trigonometric polynomials.
\end{Remark}

The following simple consequence of the above results provides a necessary condition
for the existence of UEP tight wavelet frames.

\begin{Corollary}\label{cor:UEP_matrix}
Let the trigonometric polynomial $p \in \C[T]$ satisfy  $p(\bone)=1$. For the
existence of trigonometric polynomials $q_j$ satisfying \eqref{id:UEP}, it is
necessary that the sub-QMF condition
\begin{equation}\label{id:subQMF}
1-\sum_{\sigma\in G}p^{\sigma*}p^\sigma\>\ge\>0
\end{equation}
holds on $\T^d$. In particular, it is necessary that $1-p^*p$ is non-negative on
$\T^d$.
\end{Corollary}

Next, we  give an example of a trigonometric polynomial $p$ satisfying
$p(\bone)=1$,  but for which the  corresponding
polynomial $f$ is negative for some $\omega \in \R^3$.

\begin{Example} Consider
\begin{eqnarray*}
% p(z_1,z_2,z_3)&=&e^{-3i(\omega_1+\omega_2+\omega_3)} \Big(9 \cos^2\left(\frac{\omega_1}{2}\right)\cos^2\left(\frac{\omega_2}{2}\right)
%\cos^2\left(\frac{\omega_3}{2}\right)\cos^2\left(\frac{\omega_1+\omega_2+\omega_3}{2}\right) - \\
%&& 2 \cos\left(\frac{\omega_1}{2}\right)\cos^3\left(\frac{\omega_2}{2}\right)
%\cos^3\left(\frac{\omega_3}{2}\right)\cos^3\left(\frac{\omega_1+\omega_2+\omega_3}{2} \right)\Big)
 p(z_1,z_2,z_3)&=&6z_1z_2z_3 \left(\frac{1+z_1}{2}\right)^2\left(\frac{1+z_2}{2}\right)^2\left(\frac{1+z_3}{2}\right)^2
\left(\frac{1+z_1z_2z_3}{2}\right)^2 - \\
&& \frac{5}{4}z_1 \left(\frac{1+z_1}{2}\right)\left(\frac{1+z_2}{2}\right)^3\left(\frac{1+z_3}{2}\right)^3
\left(\frac{1+z_1z_2z_3}{2}\right)^3 -\\
&&  \frac{5}{4}z_2 \left(\frac{1+z_1}{2}\right)^3\left(\frac{1+z_2}{2}\right)\left(\frac{1+z_3}{2}\right)^3
\left(\frac{1+z_1z_2z_3}{2}\right)^3 -\\
&&  \frac{5}{4}z_3 \left(\frac{1+z_1}{2}\right)^3\left(\frac{1+z_2}{2}\right)^3\left(\frac{1+z_3}{2}\right)
\left(\frac{1+z_1z_2z_3}{2}\right)^3 -\\
&&  \frac{5}{4}z_1 z_2z_3\left(\frac{1+z_1}{2}\right)^3\left(\frac{1+z_2}{2}\right)^3\left(\frac{1+z_3}{2}\right)^3
\left(\frac{1+z_1z_2z_3}{2}\right).
\end{eqnarray*}
The associated refinable function is continuous as the corresponding subdivision scheme is uniformly convergent,
but $p$ does not satisfy the sub-QMF condition, as
$$
  1-\sum_{\sigma\in G}|p^\sigma(\be^{-i \omega})|^2<0 \quad \hbox{for} \quad
  \omega=\left( \frac{\pi}{6},0,0\right).
$$
\end{Example}

\subsection{Sums of squares} \label{subsec:sumsofsquares}

Next we give another equivalent formulation of the UEP in terms of a sums of
squares problem for the
Laurent polynomial
\begin{equation}\label{def:f}
 f:=1-\sum_{\sigma\in G}p^{\sigma*}p^{\sigma}.
\end{equation}
We say that $f \in C[T]$ is a {\it sum of hermitian squares},
if there exist $h_1,\ldots,h_r\in \C[T]$ such that $\displaystyle f=\sum_{j=1}^r h_j^*h_j$.
We start with  the following auxiliary lemma.

\begin{Lemma}\label{lem:subQMFiso}%
Let $p\in \C[T]$ with isotypical components $p_\chi$, $\chi\in G'$.
\begin{itemize}
\item[(a)]
$ \displaystyle \sum_{\sigma\in G}p^{\sigma*}p^\sigma \>=\>m\cdot\sum_{\chi\in
G'} p_\chi^*p_\chi$ is a $G$-invariant Laurent polynomial in $\R[T]$.
\item[(b)]
If $f$ in \eqref{def:f} is a sum of hermitian squares
\begin{equation}\label{id:tildeH}%
f\>=\>\sum_{j=1}^rh_j^*h_j,
\end{equation}
with $h_j\in \C[T]$, then
\begin{equation}\label{id:H}%
f\>=\>\sum_{j=1}^r\sum_{\chi \in
G'}\tilde h_{j,\chi}^*\tilde h_{j,\chi},
\end{equation}
with the $G$-invariant polyphase components $\tilde h _{j,\chi}\in \C[T]$.
\end{itemize}
\end{Lemma}

\begin{proof}
Similar computations as in the proof of Theorem \ref{th:UEP_polyphase} yield
the identity in (a). The $G$-invariance and
invariance by involution are obvious.
For (b) we observe that the left-hand side of \eqref{id:tildeH} is
$G$-invariant as well. Therefore, \eqref{id:tildeH} implies
$$
  1-\sum_{\sigma\in G}p^{\sigma*}p^{\sigma}= m^{-1}\sum_{j=1}^r \sum_{\sigma\in G}
  h_j^{\sigma*}h_j^{\sigma}.
$$
Using the result in (a) we get
$$
 m^{-1}\sum_{j=1}^r \sum_{\sigma\in G}h_j^{\sigma*}h_j^{\sigma}
 \>=\>\sum_{j=1}^r\sum_{\chi\in G'}h_{j,\chi}^*h_{j,\chi}.
$$
The polyphase component $\tilde h_{j,\chi}=z^{-\alpha_\chi}h_{j,\chi}$, with
$\alpha_\chi\in\Z^d$ and $\alpha_\chi\equiv \chi$ mod~$M\Z^d$, is $G$-invariant
 and satisfies $\tilde h_{j,\chi}^*\tilde h_{j,\chi}=
h_{j,\chi}^*h_{j,\chi}$.
\end{proof}

The results in \cite{LS06} imply that having a sum of hermitian squares decomposition of
\[
   f=1-\sum_{\sigma\in G}p^{\sigma*}p^{\sigma}
   =\sum_{j=1}^rh_j^*h_j\in \R[T],
\]
with $G$-invariant polynomials $h_j\in\C[T]$, is sufficient for the existence of
the polynomials $q_j$ in Theorem~\ref{th:UEP}. The authors in \cite{LS06} also
provide a method for the construction of $q_j$ from a sum of squares
decomposition of the trigonometric polynomial $f$. Lemma \ref{lem:subQMFiso}
shows that one does not need to require $G$-invariance of $h_j$ in
\eqref{def:f}. Moreover, it is not mentioned in \cite{LS06}, that the existence
of the sos decomposition of $f$ is
 also a necessary condition, and, therefore, provides another equivalent formulation
 of the UEP conditions \eqref{id:UEP}. We state the following extension of \cite[Theorem 3.4]{LS06}.

\begin{Theorem}\label{th:LaiSt}
For any $p \in \C[T]$, with $p(\bone)=1$, the following conditions are equivalent.
 \begin{itemize}
 \item[(i)]  There exist trigonometric polynomials $h_1,\ldots,h_r \in \C[T]$  satisfying
\eqref{def:f}.
\item[(ii)] There exist trigonometric polynomials $q_1,\ldots, q_N \in
\C[T]$  satisfying \eqref{id:UEP}.
\end{itemize}
\end{Theorem}

\begin{proof}
Assume that $(i)$ is satisfied. Let $\chi_k$ be the elements of
$G'\simeq\{\chi_1,\ldots,\chi_m\}$. For $1\le j\le r$ and $1\le k\le m$, we define
the polyphase components $\tilde h_{j,\chi_k}$ of $h_j$ and set $\alpha_\chi\in\Z^d$,
$\alpha_\chi\equiv \chi$ mod~$M\Z^d$, as in Lemma
\ref{lem:subQMFiso}. The constructive method in the proof of \cite[Theorem
3.4]{LS06} yields the explicit form of $q_1,\ldots,q_N$, with $N=m(r+1)$,
satisfying \eqref{id:UEP}, namely
\begin{eqnarray}
   q_k&=&  m^{-1/2}z^{\alpha_{\chi_k}}
   (1-mpp_{\chi_k}^*), \qquad 1\le k\le m,\label{def:LSqk1}\\
   q_{mj+k}&=& p
   \tilde h_{j,\chi_k}^*
   , \qquad 1\le k\le m, \qquad 1 \le j\le r.\label{def:LSqk2}
\end{eqnarray}
Conversely, if $(ii)$ is satisfied, we obtain from \eqref{identity:UEPmatrixform}
\[
    I_m-
      \left[\begin{matrix} \vdots\\p^{\sigma*}\\ \vdots\end{matrix}
  \right]_{\sigma\in G}\cdot
  \left[\begin{matrix} \cdots & p^{\sigma}& \cdots\end{matrix}
  \right]_{\sigma\in G}= \sum_{j=1}^N
   \left[\begin{matrix} \vdots\\q_j^{\sigma*}\\ \vdots\end{matrix}
  \right]_{\sigma\in G}\cdot
  \left[\begin{matrix} \cdots & q_j^{\sigma}& \cdots\end{matrix}
  \right]_{\sigma\in G}.
\]
The determinant of the matrix on the left-hand side is equal to  $f$ in
\eqref{def:f}, and, by the Cauchy-Binet formula, the determinant of the matrix
on the right-hand side is a sum of squares.
\end{proof}

\begin{Remark}\label{rem:matrix-extension}
The constructive method in \cite{LS06} yields $N=m(r+1)$ trigonometric
polynomials $q_j$ in \eqref{id:UEP}, where $r$ is the number of trigonometric
polynomials $h_j$  in \eqref{def:f}. Moreover, the degree of some  $q_j$ in
\eqref{def:LSqk1} and \eqref{def:LSqk2} is at least
twice as high as the degree of $p$.
\end{Remark}

% ab hier erneuert 15.5.2012

Next, we give an equivalent formulation of the UEP condition in terms of
hermitian sums of squares, derived from the identities \eqref{id:equiv_UEP_1}
   in Theorem
\ref{th:UEP_polyphase}. Our goal is to improve the constructive method
in \cite{LS06} and to give an algebraic equivalent formulation
that directly delivers the trigonometric polynomials $q_j$ in Theorem~\ref{th:UEP},
avoiding any extra computations as in \eqref{def:LSqk1} and \eqref{def:LSqk2}.
To this end we write $A\>=\>\C[T]$ and consider
$$A\otimes_\C A=\C[T\times T].$$
So $A$ is the ring of Laurent polynomials in $d$ variables $z_1,\dots, z_d$. We
may identify $A\otimes A$ with the ring of Laurent polynomials in $2d$ variables
$u_1,\dots,u_d$ and $v_1,\dots,v_d$, where $u_j=z_j\otimes1$ and $v_j=1\otimes
z_j$, $j=1,\dots,d$. On $A$ we have already introduced the $G'$-grading
$A=\bigoplus_{\chi\in G'} A_\chi$ and the involution $*$ satisfying
$z_j^*=z_j^{-1}$. On $A\otimes A$ we consider the involution $*$ defined by
$(p\otimes q) ^*=q^*\otimes p^*$ for $p,q\in A$. Thus $u_j^*=v_j^{-1}$ and $v_j^*=
u_j^{-1}$. An element $f\in A\otimes A$ will be called \emph{hermitian} if
$f=f^*$. We say that $f$ is a sum of hermitian squares if there are finitely
many $q_1,\dots,q_r\in A$ with $\displaystyle f=\sum _{k=1}^rq_k^*\otimes q_k$. On $A\otimes
A$ we consider the grading
$$A\otimes A\>=\>\bigoplus_{\chi,\eta\in G'}A_\chi\otimes A_\eta.$$
So $A_\chi\otimes A_\eta$ is spanned by the monomials $u^\alpha v^\beta$ with
$\alpha+M\Z^d=\chi$ and $\beta+M\Z^d=\eta$. Note that $(A_\chi\otimes
A_\eta)^*=A_{-\eta}\otimes A_{-\chi}$.

The multiplication homomorphism $\mu\colon A\otimes A\to A$ (with $\mu(p\otimes
q)=pq$) is compatible with the involutions. Let $I= \ker(\mu)$, the ideal in
$A\otimes A$ that is generated by $u_j-v_j$ with $j=1,\dots,d$. We also need to
consider the smaller ideal
$$J\>:=\>\bigoplus_{\chi,\eta\in G'}\Bigl(I\cap\bigl(A_\chi\otimes
A_\eta\bigr)\Bigr)$$ of $A\otimes A$. The ideal $J$ is $*$-invariant. Note that
the inclusion $J\subset I$ is proper since, for example, $u_j-v_j\notin J$.

\begin{Theorem}\label{th:UEP_hermitian}
Let $p\in A=\C[T]$ satisfy $p(\bone)=1$. The following conditions are
equivalent.
\begin{itemize}
\item[(i)] The Laurent polynomial
$$f=1-\sum_{\sigma\in G}p^{\sigma*}p^\sigma$$
 is a sum of hermitian squares in
$A$; that is, there exist $h_1,\dots,h_r\in A$ with $\displaystyle f= \sum_{j=1}^r h_j^*h_j$.
\item[(ii)]
For any hermitian elements $h_\chi=h_\chi^*$ in $A_{-\chi}\otimes A_{\chi}$,
with $\mu(h_\chi)=\frac1m$ for all $\chi\in G'$, the element
$$g\>:=\>\sum_{\chi\in G'}h_\chi-p^*\otimes p$$
is a sum of hermitian squares in $A\otimes A$ modulo~$J$; that is,
there exist $q_1,\dots,q_N\in A$ with $\displaystyle g- \sum_{j=1}^N q_j^*q_j\in J$.
\item[(iii)]
$p$ satisfies the UEP condition \eqref{id:UEP} for suitable $q_1, \dots,q_N\in
A$.
\end{itemize}
\end{Theorem}

\begin{proof}
By Theorem \ref{th:LaiSt}, $(i)$ is equivalent to  $(iii)$.
In $(ii)$, let hermitian elements $h_\chi\in A_{-\chi,\chi}$ be given with
$\mu(h_\chi)=\frac1m$. Then $(ii)$ is equivalent to the existence of
$q_1,\dots,q_N\in A$ with
\begin{equation}\label{id:gmodJ}
   \sum_{\chi\in G'}h_\chi-p^*\otimes p-\sum_{j=1}^Nq_j^*\otimes q_j\in J.
\end{equation}
We write
$\displaystyle p=\sum_{\chi \in G'}p_\chi$ and $q_j$
as the sum of its isotypical components and observe that \eqref{id:gmodJ} is
equivalent to
\begin{equation}\label{id:gchieta}
   \delta_{\chi,\eta}h_\chi -p_\chi^*\otimes p_\eta-
   \sum_{j=1}^Nq_{j,\chi}^*\otimes q_{j,\eta}
   \in  I\quad\hbox{for all}\quad \chi,\eta\in G'.
\end{equation}
Due to $\mu(h_\chi)=\frac1m$,  the relation \eqref{id:gchieta} is
an equivalent reformulation of equations \eqref{id:equiv_UEP_1} in Theorem
\ref{th:UEP_polyphase}, and therefore equivalent to equations
\eqref{id:UEP}.
\end{proof}

\begin{Remark}\label{rem:UEP_hermitian}
\begin{itemize}
\item[$(i)$]
The proof of Theorem \ref{th:UEP_hermitian} does not
depend on the choice of the hermitian elements $h_\chi\in
A_{-\chi}\otimes A_{\chi}$ in $(ii)$.  Thus, it suffices to choose
particular hermitian elements satisfying  $\mu(h_\chi)=\frac1m$. For example,
if $p_\chi(\bone)=m^{-1}$ is satisfied for all $\chi\in G'$, we can choose
\begin{equation}\label{def:hchi}
   h_\chi = \sum_{\alpha\equiv \chi}{\rm Re}(p(\alpha))u^{-\alpha}v^\alpha,
\end{equation}
where $p(\alpha)$ are the coefficients of the Laurent polynomial $p$.
\item[$(ii)$]
The same Laurent polynomials $q_1,\ldots,q_N$ can be chosen
in Theorem~\ref{th:UEP_hermitian}
$(ii)$ and $(iii)$. This is the main advantage of working with the
condition $(ii)$ rather than with $(i)$.
\end{itemize}
\end{Remark}

%%%%%%%%%%%%%%%%%%%%%%%%%%%%%%%%%%%%%%%%%%%%%%
  \subsection{Semi-definite programming}  \label{subsec:semi-definite}
%%%%%%%%%%%%%%%%%%%%%%%%%%%%%%%%%%%%%%

We next devise a constructive method for determining the Laurent polynomials
$q_j$ in \eqref{id:UEP}. This method is based on $(ii)$ of Theorem \ref{th:UEP_hermitian}
and $(i)$ of Remark \ref{rem:UEP_hermitian}.

For a Laurent polynomial $ p=\sum_\alpha p(\alpha)z^\alpha$,
let $\cN\subset \Z^d$ contain $\{\alpha \in \Z^d \ : \ p(\alpha) \not=0\}$.
We also define the tautological (column) vector
$$
 \bx=\left[z^\alpha \ : \ \alpha \in \cN \right]^T,
$$
and the orthogonal projections $E_\chi \in \R^{|\cN| \times |\cN|}$  to be diagonal matrices with diagonal
entries given by
$$
 E_\chi(\alpha,\alpha)=\left\{ \begin{array}{cc} 1, & \alpha \equiv \chi \ \hbox{mod} \ M\Z^d, \\
 0, & \hbox{otherwise}, \end{array}\right. \alpha \in \cN.
$$

\begin{Theorem} \label{th:UEP_semidefinite}
Let
\begin{equation}\label{def:p_coeff}
    p=\bp\cdot \bx\in A=\C[T], \qquad \bp=[p(\alpha) \ : \ \alpha \in \cN]\in \C^{|\cN|},
\end{equation}
satisfy $p_\chi(\bone)=m^{-1}$ for all $\chi\in G'$. The following conditions
are equivalent.
\begin{itemize}
\item[(i)] There exist row vectors $\bq_j=[q_j(\alpha) \ : \ \alpha \in \cN]\in \C^{|\cN|}$,
 $1 \le j \le N$, satisfying the identities
\begin{eqnarray} \label{id:equiv_UEP_2}
  \bx^*E_\chi \left( {\rm diag}({\rm Re}\,\bp)-\bp^* \bp  -\sum_{j=1}^N \bq_j^* \bq_j  \right) E_\eta
  \bx=0\quad\hbox{for all}\quad
 \chi,\eta \in G'.
\end{eqnarray}
\item[(ii)]
$p$ satisfies the UEP condition \eqref{id:UEP} with
$$
    q_j=\bq_j\cdot \bx\in \C[T], \qquad j=1,\ldots,N,
$$
and suitable row vectors $\bq_j\in \C^{|\cN|}$.
\end{itemize}
\end{Theorem}

\begin{proof}
Define
$$
 \bv=\left[1\otimes z^\alpha \ : \ \alpha \in \cN \right]^T \in (A\otimes A)^{|\cN|}.
$$
Note that $\bp\bv=1\otimes p$ and the definition of $E_\chi$ gives
$\bp E_\chi\bv=1\otimes p_\chi$. Therefore, we have
$$
 \bv^* E_\chi \bp^* \bp E_\eta \bv=p_\chi^*\otimes p_\eta
 \quad\hbox{for all}\quad \chi,\eta\in G',
$$
and the analogue for $q_{j,\chi}^*\otimes q_{j,\eta}$. Moreover, we have
$$
 \bv^* E_\chi \hbox{diag}({\rm Re}\,\bp) E_\eta \bv=\delta_{\chi,\eta}
 \sum_{\alpha\equiv \chi}{\rm Re}(p(\alpha))u^{-\alpha}v^\alpha.
$$
Due to
$p_\chi(\bone)=m^{-1}$ and by Remark \ref{rem:UEP_hermitian} we choose
$h_\chi=\bv^* E_\chi \hbox{diag}({\rm Re}\, \bp) E_\chi \bv$
as the hermitian elements in
Theorem \ref{th:UEP_hermitian}$(ii)$, and
 the relation \eqref{id:gchieta} is equivalent to
$$
  \bv^*E_\chi \left( \hbox{diag}({\rm Re}\,\bp)-\bp^* \bp
  -\sum_{j=1}^N \bq_j^* \bq_j  \right) E_\eta   \bv \in I
  \quad\hbox{for all}\quad \chi,\eta\in G'.
$$
Due to $\mu(\bv)=\bx$, the claim follows from the equivalence of $(ii)$ and $(iii)$ in Theorem \ref{th:UEP_hermitian}.
\end{proof}

We suggest the following constructive method based on Theorem \ref{th:UEP_semidefinite}.
Given the trigonometric polynomial $p$ and the vector $\bp$ in \eqref{def:p_coeff},
define the matrix
\begin{equation} \label{def:R}
 R=\hbox{diag}({\rm Re}\,\bp)-\bp^*\bp\in \C^{|\cN| \times |\cN|}.
 %\quad \hbox{and} \quad S=\sum_{j=1}^N\bq_j^* \bq_j.
\end{equation}
Then the task of constructing tight wavelet frames can be formulated as the
following problem of {\bf semi-definite programming}:  find a
matrix $O\in \C^{|\cN| \times |\cN|}$ such that
\begin{equation}\label{id:Sposdef}
   S:=R+O\quad\hbox{is positive semi-definite}
\end{equation}
subject to the constraints
\begin{equation}\label{id:null-matrices}
  \bx^* E_\chi \, O  \, E_\eta  \bx=0  \quad \hbox{for all}\quad
  \chi, \eta \in G'.
\end{equation}
If such a matrix $O$ exists,
we determine the trigonometric polynomials $q_j=\bq_j \bx\in \C[T]$
by choosing any decomposition of the form
$$
    S=\sum_{j=1}^N\bq_j^* \bq_j
$$
with standard methods
from linear algebra.

\medskip

If the semi-definite programming problem does not have a
solution, we can increase the set $\cN$ and start all over.
Note that the identities \eqref{id:null-matrices} are equivalent to the
following linear constraints on the null-matrices $O$
$$
 \sum_{\alpha \equiv \chi, \beta \equiv \eta}
 O_{\alpha,\beta} z^{\beta-\alpha}=0 \quad\hbox{for all}\quad \chi, \eta
 \in G',
$$
or, equivalently,
$$
 \sum_{\alpha \equiv \chi} O_{\alpha,
 \alpha+\tau}=0 \quad \hbox{for all}\quad
 \tau \in \{\beta-\alpha \ : \ \alpha, \beta \in
 \cN\}.
$$

\begin{Example} To illustrate the concept of null-matrices, we consider first a very
prominent one-dimensional example of a Daubechies wavelet. Let
$$
 p=\bp \cdot \bx, \quad
 \bp=\frac{1}{8} \left[\begin{array}{cccc} 1+\sqrt{3} & 3+\sqrt{3} & 3-\sqrt{3}& 1-\sqrt{3}
 \end{array}\right],
$$
and $\bx=\left[1,z,z^2,z^3\right]^T$. In this case $M=m=2$, $G\simeq\{0,\pi\}$,
$G'\simeq\{0,1\}$ and the orthogonal projections $E_\chi \in \R^{4 \times
4}$, $\chi \in G'$, are given by
$$
 E_0=\hbox{diag}[1,0,1,0] \quad \hbox{and} \quad E_1=\hbox{diag}[0,1,0,1].
$$
By \eqref{def:R}, we have
$$
 R=\frac{1}{64} \left[\begin{array}{rrrr}4+6\sqrt{3}& -6-4\sqrt{3}& -2\sqrt{3}&2\\
 -6-4\sqrt{3}& 12+2\sqrt{3} & -6& 2\sqrt{3} \\ -2\sqrt{3} & -6 & 12-2\sqrt{3} & -6+4\sqrt{3} \\
 2& 2\sqrt{3} & -6+4\sqrt{3} & 4-6\sqrt{3} \end{array} \right],
$$
which is not positive semi-definite. Define
$$
  O=\frac{1}{64} \left[\begin{array}{rrrr}-8\sqrt{3}& 8\sqrt{3}& 0&0\\
  8\sqrt{3}& -8\sqrt{3} & 0& 0 \\ 0 & 0 & 8\sqrt{3} & -8\sqrt{3} \\
 0& 0 & -8\sqrt{3} & 8\sqrt{3} \end{array} \right]
$$
satisfying \eqref{id:null-matrices}. Then $S=R+O$ is positive semi-definite, of
rank one, and yields the well-known Daubechies wavelet, see \cite{Daub} defined
by
$$
 q_1= \frac{1}{8} \left[ \begin{array}{cccc}
 1-\sqrt{3} & -3+\sqrt{3} & 3+\sqrt{3}& -1-\sqrt{3}
 \end{array}\right] \cdot \bx.
$$
%This example also illustrates that we cannot expect that appropriate
%modifications of the diagonal entries of $R$ will produce a desired matrix $S$.
%Note that for the null-matrices $O=\hbox{diag}[0,-h,0,h]$, the determinant of
%$R+O$ is equal to $-4 \sqrt{3} h^2$. Thus, $R+O$ is positive semi-definite only
%if $h=0$.
\end{Example}

Another two-dimensional example of one possible choice of an appropriate
null-matrix satisfying \eqref{id:null-matrices} is given in Example
\ref{ex:butterfly}.

\begin{Remark}
Another, very similar, way of working with null-matrices was pursued already in
\cite{CS08}.
\end{Remark}

%%%%%%%%%%%%%%%%%%%%%%%%%%%%%%%%%%%%%%%%%%%%%%%%%%%%%%%%%%
\section{Existence and constructions of tight wavelet frames} \label{sec:algebra}

In this section we use results from algebraic geometry and Theorem \ref{th:LaiSt}
to resolve the problem of existence of tight wavelet frames. Theorem  \ref{th:LaiSt}
allows us to reduce the problem of existence of $q_j$ in \eqref{id:UEP} to the problem
of existence of an sos decomposition of a single nonnegative polynomial
\begin{equation}
 f=1-\sum_{\sigma \in G} p^{\sigma*} p^{\sigma} \in \R[T]. \notag
\end{equation}
In subsection \ref{subsec:existence}, for dimension $d=2$, we show that the
polynomials $q_1,\dots,q_N  \in \C[T]$ as in Theorem \ref{th:UEP} always exist.
This result is based on recent progress in real algebraic geometry.
We also include an example of a three-dimensional trigonometric polynomial $p$, satisfying the
sub-QMF condition \eqref{id:subQMF}, but for which
trigonometric polynomials $q_1,\ldots,q_N$ as in Theorem \ref{th:UEP} do not exist.
In subsection \ref{subsec:sufficient}, we give sufficient conditions for the
existence of the $q_j$'s in the multidimensional case and give several explicit
constructions of tight wavelet frames in section \ref{subsec:construction}.

%%%%%%%%%%%%%%%%%%%%%%%%%%%%%%%%%%%%%%%%%%%%%%%%%%%%%%

\subsection{Existence of tight wavelet frames} \label{subsec:existence}

%%%%%%%%%%%%%%%%%%%%%%%%%%%%%%%%%%%%%%%%%%%%%%%%%%%%%%

In this section we show that in the two-dimensional case ($d=2$) the question of
existence of a wavelet tight frame can be positively answered using the results
from \cite{sch:surf}. Thus, Theorem \ref{th:existence2dim}  answers a long
standing open question about the existence of tight wavelet frames as in Theorem \ref{th:UEP}.
The result of Theorem \ref{th:noUEP} states that in the dimension $d \ge 3$
for a given trigonometric polynomial $p$ satisfying $p(\bone)=1$ and the sub-QMF condition \eqref{id:subQMF}
one cannot always determine trigonometric polynomials $q_j$ as in
Theorem \ref{th:UEP}.

\begin{Theorem}\label{th:existence2dim}%
Let $d=2$, $p\in\C[T]$ satisfy $p(\bone)=1$ and
$\displaystyle \sum_{\sigma\in G}p^{\sigma *} p^\sigma \le1$ on $\T^2=T(\R)$. Then there exist $N\in\N$
and trigonometric polynomials $q_1,\dots,q_N\in\C[T]$ satisfying
\begin{equation}\label{aux:th:existence2dim}%
 \delta_{\sigma,\tau}\>=\>p^{\sigma*}p^{\tau}+\sum_{j=1}^N
 q_j^{\sigma*}q_j^{\tau},\qquad \sigma,\,\tau\in G.
\end{equation}
\end{Theorem}

\begin{proof}
The torus $T$ is a non-singular affine algebraic surface over $\R$, and $T(\R)$
is compact. The polynomial $f$ in \eqref{def:f} is in $\R[T]$  and is
nonnegative on $T(\R)$ by assumption. By Corollary~3.4 of \cite{sch:surf}, there
exist $h_1,\dots,h_r\in\C[T]$ satisfying $\displaystyle f=\sum_{j=1}^r h_j^* h_j$. According to
Lemma \ref{lem:subQMFiso} part $(b)$, the polynomials $h_j$ can be taken to be
$G$-invariant. Thus, by Theorem \ref{th:LaiSt}, there exist
polynomials $q_1,\dots,q_N$ satisfying \eqref{aux:th:existence2dim}.
\end{proof}

The question may arise, if there exists
a trigonometric polynomial $p$ that satisfies $p(\bone)=1$ and  the sub-QMF condition
$\displaystyle \sum_{\sigma\in G} p^{\sigma*}p^\sigma \le 1$ on $\T^d$, but for which there exists no UEP tight frame
as in Theorem \ref{th:UEP}.
Or, due to Corollary \ref{cor:UEP_matrix}, if we can find such a $p$, for which the nonnegative
trigonometric polynomial $1-p^{*}p$ is not a sum of hermitian squares of trigonometric polynomials?

\begin{Theorem}\label{th:noUEP}
There exists $p\in\C[T]$ satisfying $p(\bone)=1$ and the sub-QMF condition on $\T^3$,
such that $1-p^*p$ is not a sum of hermitian squares in $\R[T]$.
\end{Theorem}

The proof is constructive. The following example defines a family of
trigonometric polynomials with the properties stated in Theorem \ref{th:noUEP}.
We make use of the following local-global result from algebraic geometry: if
the Taylor expansion of  $f\in \R[T]$
at one of its roots  has, in local coordinates,
a homogeneous part of lowest degree which is not sos of real algebraic
polynomials, then $f$ is not sos in $\R[T]$.

\begin{Example}
Denote $z_j=e^{-i\omega_j}$, $j=1,2,3$.
We let
\[
   p(z)=\Big(1-c \cdot m(z) \Big) a(z),\quad z \in T, \quad 0<c\le \frac{1}{3},
\]
where
\[
   m(z)= y_1^4y_2^2+y_1^2y_2^4+y_3^6-3y_1^2y_2^2y_3^2 \in \R[T],\qquad
   y_j=\sin\omega_j.
\]
In the local coordinates $(y_1,y_2,y_3)$ at $z=\bone$,
$m$ is the well-known Motzkin polynomial in $\R[y_1,y_2,y_3]$; i.e.
 $m$ is not sos in $\R[y_1,y_2,y_3]$. Moreover, $a\in \R[T]$ is chosen
 such that
\begin{equation}\label{eq:propA}
   D^\alpha a(\bone)=\delta_{0,\alpha},\quad
   D^\alpha a(\sigma)=0,\quad 0\le |\alpha|< 8, \quad
   \sigma\in G\setminus \{\bone\},
\end{equation}
and $\displaystyle \sum_{\sigma \in G} a^{\sigma*}a^\sigma \le 1$. Such $a$ can be, for example, any scaling
symbol of a 3-D orthonormal wavelet with 8 vanishing moments; in particular, the
tensor product Daubechies symbol $a(z)=m_8(z_1) m_8(z_2) m_8(z_3)$ with $m_8$ in
\cite{Daub} satisfies  conditions \eqref{eq:propA} and $\displaystyle \sum_{\sigma \in G} a^{\sigma *}a^\sigma
= 1$. The properties of $m$ and $a$ imply that
\begin{itemize}
\item[1.] $p$ satisfies the sub-QMF condition on $\T^3$, since $m$ is $G$-invariant and
$0\le 1-c \cdot m \le 1$ on $\T^3$,
\item[2.] $p$ satisfies sum rules of order at least $6$,
\item[3.] the Taylor expansion of
$1-p^*p$ at $z=\bone$, in local coordinates $(y_1,y_2,y_3)$,
has  $2\cdot c \cdot m$ as its homogeneous part of lowest degree.
\end{itemize}
Therefore, $1-p^*p$ is not sos of trigonometric polynomials in $\R[T]$. By Corollary \ref{cor:UEP_matrix}, the corresponding
nonnegative trigonometric polynomial  $f$ in \eqref{def:f} has no sos decomposition.
\end{Example}

%%%%%%%%%%%%%%%%%%%%%%%%%%%%%%%%%%%%%%%%%%%%%%%%%%
\subsection{Sufficient conditions for existence of tight wavelet frames} \label{subsec:sufficient}
%%%%%%%%%%%%%%%%%%%%%%%%%%%%%%%%%%%%%%%%%%%%%%%%%%

In the general multivariate case $d \ge 2$, in Theorem \ref{th:existence-ddim},
we provide a sufficient condition for the existence of a sums of squares
decomposition of $f$ in \eqref{def:f}. This condition is based on the properties
of the Hessian of $f \in \R[T]$
$$
  \Hess(f)=\left( D^\mu f  \right)_{\mu \in \N_0^s, |\mu|=2},
$$
where $f$ is a trigonometric polynomial in $\omega \in \R^d$ and $D^\mu$ denotes the $|\mu|-$th partial
derivative with respect to $\omega \in \R^d$.

\begin{Theorem}\label{hessiancrit}%
Let $V$ be a non-singular affine $\R$-variety for which $V(\R)$ is
compact, and let $f\in\R[V]$ with $f\ge0$ on $V(\R)$. For every $\xi
\in V(\R)$ with $f(\xi)=0$, assume that the Hessian of $f$ at $\xi$
is strictly positive definite. Then $f$ is a sum of squares in
$\R[V]$.
\end{Theorem}

\begin{proof}
The hypotheses imply that $f$ has only finitely many zeros in
$V(\R)$. Therefore the claim follows from \cite[Corollary
2.17, Example 3.18]{sch:mz}.
\end{proof}

Theorem \ref{hessiancrit} implies the following result.

\begin{Theorem}\label{th:existence-ddim}%
Let $p\in\C[T]$ satisfy $p(\bone)=1$ and $\displaystyle f=1-\sum_{\sigma\in G}p^{\sigma*}p^\sigma \ge 0$ on $T(\R)=\T^d$.  If the
Hessian of $f$ is positive
definite at every zero of $f$ in $\T^d$, then there exist $N\in\N$ and polynomials
$q_1,\dots,q_N\in\C[T]$ satisfying \eqref{id:UEP}.
\end{Theorem}

\begin{proof}
By Theorem \ref{hessiancrit}, $f$ is a sum of squares in $\R[T]$. The claim
follows then by Theorem \ref{th:existence2dim}.
\end{proof}

Due to $p(\bone)=1$, $z=\bone$ is obviously a zero of $f$. We show next how to
express the Hessian of $f$ at $\bone$ in terms of the gradient $\nabla p(\bone)$
and the Hessian of $p$ at $\bone$, if $p$ additionally satisfies the so-called
sum rules of order $2$, or, equivalently, satisfies the zero conditions of order
$2$. We say that $p \in \C[T]$ satisfies zero conditions of order $k$, if
$$
  D^\mu p(\be^{-i\sigma})=0, \quad \mu \in \N_0^d, \quad |\mu|<k, \quad \sigma \in G\setminus\{0\},
$$
see \cite{JePlo,JiaJiang} for details. The assumption that $p$ satisfies sum
rules of order $2$ together with $p(\bone)=1$ are necessary for the continuity
of the corresponding refinable function $\phi$.

\begin{Lemma}\label{lem:Hessianf_HessianP}%
Let $p\in\C[T]$ with real coefficients satisfy the sum rules of order $2$ and
$p(\bone)=1$. Then the Hessian of $\displaystyle f=1-\sum_{\sigma\in G}p^{\sigma*}p^\sigma$ at
$\bone$ is equal to
$$-2\,\Hess(p)(\bone)-2\,\nabla p(\bone)^*\nabla p(\bone).$$
\end{Lemma}

\begin{proof}
We expand the trigonometric polynomial $p$ in a neighborhood of
 $\bone$ and get
$$
    p(\be^{-i\omega})=1+ \nabla p(\bone) \omega+
    \frac{1}{2}\omega^T \mbox{Hess}(p)(\bone) \omega
    +\mathcal{O}(|\omega|^3).
$$
Note that, since the coefficients of $p$ are real, the row vector $v=\nabla
p(\bone)$ is purely imaginary and $\mbox{Hess}(p)(\bone)$ is real and symmetric.
The sum rules of order $2$ are equivalent to
$$
    p^\sigma(\bone)= 0,\quad
    \nabla p^\sigma(\bone)=0\qquad
    \mbox{for all}\quad
    \sigma\in G\setminus\{ 0\}.
$$
Thus, we have $p^\sigma(\be^{-i\omega})=\mathcal{O}(|\omega|^2)$ for all
$\sigma\in G\setminus\{0\}$.  Simple computation yields
\begin{eqnarray*}
  |p(\be^{-i\omega})|^2&=&1+ (v+\ol{v}) \omega+ \omega^T (\mbox{Hess}(p)(\bone)+v^*v) \omega
    +\mathcal{O}(|\omega|^3) \\&=&1+  \omega^T (\mbox{Hess}(p)(\bone)+v^*v) \omega
    +\mathcal{O}(|\omega|^3).
\end{eqnarray*} Thus, the claim follows.
\end{proof}

\begin{Remark} Note that $\Hess(f)$ is a zero matrix, if
$p$ is a symbol of interpolatory subdivision scheme, i.e.,
$$
 p=m^{-1}+ m^{-1}\sum_{\chi \in G' \setminus \{0\}} p_\chi,
$$
and $p$ satisfies zero conditions of order at least $3$. This property of $\Hess(f)$
follows directly from the equivalent formulation of zero conditions of order
$k$, see \cite{Cabrelli}. The examples of $p$ with such properties are for
example the butterfly scheme in Example \ref{ex:butterfly} and  the
three-dimensional interpolatory scheme in Example \ref{ex:3D_butterfly}.
\end{Remark}

\begin{Remark}
The sufficient condition of Theorem \ref{hessiancrit} can be
generalized to cases when the order of vanishing of $f$ is larger than two.
Namely, let $V$ and $f$ as in \ref{hessiancrit}, and let $\xi\in
V(\R)$ be a zero of $f$. Fix a system $x_1,\dots,x_n$ of local
(analytic) coordinates on $V$ centered at $\xi$. Let $2d>0$ be the
order of vanishing of $f$ at $\xi$, and let $F_\xi(x_1,\dots,x_n)$ be
the homogeneous part of degree $2d$ in the Taylor expansion of $f$
at $\xi$. Let us say that $f$ is \emph{strongly sos at~$\xi$} if
there exists a linear basis $g_1,\dots,g_N$ of the space of
homogeneous polynomials of degree $d$ in $x_1,\dots,x_n$ such that
$F_\xi=g_1^2+\cdots+g_N^2$. (Equivalently, if $F_\xi$ lies in the
interior of the sums of squares cone in degree~$2d$.) If $2d=2$, this
condition is equivalent to the Hessian of $f$ at $\xi$ being positive
definite.

Then the following holds: If $f$ is strongly sos at each of its zeros
in $V(\R)$, $f$ is a sum of squares in $\R[V]$. For a proof we refer
to \cite{sch:future}. As a result, we get a corresponding
generalization of Theorem \ref{th:existence-ddim}: If $f$ as in
\ref{th:existence-ddim} is strongly sos at each of its zeros in
$\T^d$, then the conclusion of \ref{th:existence-ddim} holds.
\end{Remark}

For simplicity of presentation, we start by applying the result of Theorem
\ref{th:existence-ddim} to the $2-$dimensional polynomial $f$ derived from the symbol of the three-directional
piecewise linear box spline. This example also motivates the statements of  Remark
\ref{remark:box_and_cosets_sec2.2}.

\begin{Example} \label{example:b111-sec2.2}
The three-directional piecewise linear box spline is defined by its associated trigonometric
polynomial
$$
   p(\be^{-i\omega})= e^{-i(\omega_1+\omega_2)} \cos\left(\frac{\omega_1}{2}\right)
   \cos\left(\frac{\omega_2}{2}
   \right)\cos\left(\frac{\omega_1+\omega_2}{2}\right), \quad
   \omega \in \R^2.
$$
Note that
$$
    \cos\left(\frac{\omega_1}{2}\right)
   \cos\left(\frac{\omega_2}{2}
   \right)\cos\left(\frac{\omega_1+\omega_2}{2}\right)
    =1-\frac{1}{8}\omega^T \left(\begin{matrix} 2&1\\1&2\end{matrix}\right) \omega +
    \mathcal{O}(|\omega|^4).
$$
Therefore, as the trigonometric polynomial $p$ satisfies sum rules
of order $2$, we get
$$
 f(\be^{-i\omega})=\frac{1}{8}\omega^T \left(\begin{matrix} 2&1\\1&2\end{matrix}\right) \omega +
    \mathcal{O}(|\omega|^4).
$$
Thus, the Hessian of $f$ at $\bone$ is positive definite.

To determine the other zeroes of $f$, by Lemma \ref{lem:subQMFiso} part $(a)$, we
can use either one of the representations
\begin{eqnarray*}
 f(\be^{-i\omega})&=&1-\sum_{\sigma \in \{0,\pi\}^2} \prod_{\theta \in \{0,1\}^2 \setminus \{0\}}
 \cos^2\left(\frac{(\omega+ \sigma)\cdot \theta}{2}\right)\\
 &=& \frac{1}{4} \sum_{\chi \in \{0,1\}^2}(1-\cos^2(\omega \cdot \chi)).
\end{eqnarray*}
It follows that the zeros of $f$ are the points $\omega \in
\pi \Z^2$ and, by periodicity of $f$ with period $\pi$ in both coordinate directions,
we get that
$$
   \mbox{Hess}(f)(\be^{-i\omega})=\mbox{Hess}(f)(\bone), \quad \omega \in
\pi \Z^2,
$$
is positive definite at all zeros of $f$.
\end{Example}

\begin{Remark} \label{remark:box_and_cosets_sec2.2}

\noindent $(i)$ The result of \cite[Theorem 2.4]{CS07} implies the existence of
tight frames for multivariate box-splines. According to the notation in
\cite[p.~127]{deBoor}, the corresponding trigonometric polynomial is given by
$$
  p(\be^{-i\omega})=\prod_{j=1}^n \frac{1+\be^{-i\omega \cdot \xi^{(j)}}}{2},
  \quad \omega \in \R^d,
$$
where $\Xi=(\xi^{(1)},\ldots,\xi^{(n)})\in\Z^{d\times n}$ is unimodular and has
rank $d$. (Unimodularity means that all $d\times d$-submatrices have determinant
$0,1$, or $-1$.) Moreover, $\Xi$ has the property that leaving out any  column
$\xi^{(j)}$ does not reduce its rank. (This property guarantees continuity of
the box-spline and that the corresponding polynomial $p$ satisfies at least sum
rules of order $2$.) Then one can show that
$$
  f= 1-\sum_{\sigma \in G}p^{\sigma*}p^\sigma \ge  0 \quad \hbox{on} \ \T^d,
$$
the zeros of $f$ are at $\omega \in \pi \Z^d$ and the Hessian of $f$ at
these zeros is positive definite. This yields an alternative proof for
\cite[Theorem 2.4]{CS07} in the case of box splines.

\noindent $(ii)$ If the summands $m^{-2}-p_\chi^* p_\chi$ are
nonnegative on $\T^d$, then it can be easier to determine the
zeros of $f$ by determining the common zeros of all of these summands.
\end{Remark}

\begin{Example} \label{ex:3D_butterfly}
There was an attempt to define an interpolatory scheme for 3D-subdivision with dilation matrix $2I_3$
in \cite{CMQ}. There are several inconsistencies in this paper and we give a correct description of the trigonometric polynomial
$p$, the so-called subdivision mask. Note that the scheme  we present is an extension of the 2-D
butterfly scheme to 3-D data in the following sense: if the data are constant along
one of the coordinate directions (or along the main diagonal in $\R^3$), then
the subdivision procedure keeps this property  and is identical with the
2-D butterfly scheme.

We describe the trigonometric polynomial $p$ associated with this 3-D scheme
by defining its isotypical components.  The  isotypical components,
in terms of $z_k=e^{-i\omega_k}$, $k=1,2$, are given by
\small
\begin{eqnarray*}
   p_{0,0,0}(z_1,z_2,z_3)&=&1/8,\\[12pt]
   p_{1,0,0}(z_1,z_2,z_3) &=& \frac{1}{8} \cos\omega_1 +
   \frac{\lambda}{4} \Big(\cos(\omega_1+2\omega_2)+\cos(\omega_1+2\omega_3)\\&+&\cos(\omega_1+2\omega_2+2\omega_3)\Big)
   -  \frac{\lambda}{4} \Big(\cos(\omega_1-2\omega_2)+\cos(\omega_1-2\omega_3)\\&+&
   \cos(3\omega_1+2\omega_2+2\omega_3)\Big),\\[12pt]
     p_{0,1,0}(z_1,z_2,z_3) &=&    p_{1,0,0}(z_2,z_1,z_3),\
     p_{0,0,1}(z_1,z_2,z_3) =    p_{1,0,0}(z_3,z_1,z_2),\\
     p_{1,1,1}(z_1,z_2,z_3) &=&    p_{1,0,0}(z_1z_2z_3,z_1^{-1},z_2^{-1}),\\[12pt]
     p_{1,1,0}(z_1,z_2,z_3) &=& \Big(\frac{1}{8}-\lambda \Big) \cos(\omega_1+\omega_2) +
     \lambda \Big(\cos(\omega_1-\omega_2)+\cos(\omega_1+\omega_2+2\omega_3)\Big)\\
    &-&\frac{\lambda}{4} \Big(\cos(\omega_1-\omega_2+2\omega_3)+\cos(\omega_1-\omega_2-2\omega_3)+ \\
   && \cos(3\omega_1+\omega_2+2\omega_3)+\cos(\omega_1+3\omega_2+2\omega_3)\Big),\\[12pt]
      p_{1,0,1}(z_1,z_2,z_3) &=&    p_{1,1,0}(z_1,z_3,z_2),\qquad
     p_{0,1,1}(z_1,z_2,z_3) =    p_{1,0,0}(z_2,z_3,z_1),
\end{eqnarray*}
\normalsize where $\lambda$ is the so-called tension parameter.

The polynomial $p$ also satisfies
$$
  p(z_1,z_2,z_3)=\frac{1}{8} (1+z_1)(1+z_2)(1+z_3)(1+z_1z_2z_3)q(z_1,z_2,z_3), \quad q(\bone)=1,
$$
which implies sum rules of order $2$.

\begin{itemize}
\item[$(a)$] For $\lambda=0$, we have
$q(z_1,z_2,z_3)=1/(z_1z_2z_3)$. Hence, $p$ is the scaling symbol of the trivariate
 box spline with the  direction set $(1,0,0)$, $(0,1,0)$, $(0,0,1)$, $(1,1,1)$ and whose support center
is shifted to the origin.

\item[$(b)$] For $0\le \lambda< 1/16$, the corresponding subdivision scheme converges and
has a continuous limit function. The only zeros of the associated nonnegative
trigonometric polynomial $f$ are at $\pi \Z^3$, and the Hessian of $f$
at these zeros is given by
$$
  \hbox{Hess}(f)(\bone)=\hbox{Hess}(f)(\be^{-i\omega})=\left( \begin{array}{ccc} 1-16\lambda &\frac{1}{2}-8\lambda&\frac{1}{2}-8\lambda\\
  \frac{1}{2}-8\lambda&1-16\lambda&\frac{1}{2}-8\lambda
\\\frac{1}{2}-8\lambda&\frac{1}{2}-8\lambda&1-16\lambda \end{array}\right)
$$
for all  $\omega \in \pi \Z^3$.
The existence of the sos decomposition of $f$ is guaranteed by Theorem \ref{th:existence-ddim} and
one possible decomposition of $f$ is computed as follows.

 \begin{itemize}
 \item[$(b_1)$] Denote $u:=\cos(\omega_1+\omega_2)$, $v:=\cos(\omega_1+\omega_3)$,
 $w:=\cos(\omega_2+\omega_3)$, and $\tilde u:=\sin(\omega_1+\omega_2)$, $\tilde v:=\sin(\omega_1+\omega_3)$,
 $\tilde w:=\sin(\omega_2+\omega_3)$.
 Simple computations yield
 \[
    p_{1,1,0} =\frac18 -(1-u)(\frac18-\lambda v^2-\lambda w^2)-\lambda (v-w)^2,
 \]
and
\begin{eqnarray*}
    \frac{1}{64}&-&|p_{1,1,0}|^2=
    \lambda^2 (v^2-w^2)^2 + \Big(\Big(\frac1{16}-\lambda v^2\Big) \\&+&
    \Big(\frac1{16}-\lambda w^2\Big)\Big)
    \left(\frac1{8}\tilde u^2+\lambda (v-uw)^2+\lambda(w-uv)^2\right).
\end{eqnarray*}
Therefore, $\frac{1}{64}-|p_{1,1,0}|^2$ has an sos decomposition with $7$ summands $h_j$, and each $h_j$ has
only one nonzero isotypical component.

\item[$(b_2)$] The  isotypical component $p_{1,0,0}$ is not bounded by $1/8$;
consider, for example,
$p_{1,0,0}(\be^{-i\omega})$ at the point $\omega=\left( -\frac{\pi}{6}, -\frac{2\pi}{3}, -\frac{2\pi}{3}\right)$.
Yet we obtain, by simple computations,
\[
   p_{1,0,0} = \frac18\cos\omega_1+\frac{\lambda}{2} A \sin\omega_1,\
    A:=\sin 2(\omega_1+\omega_2+\omega_3)-\sin 2\omega_2-\sin 2\omega_3,
\]
and
\[
    \frac{1}{16}-|p_{1,0,0}|^2-|p_{0,1,0}|^2-|p_{0,0,1}|^2-|p_{1,1,1}|^2 =
    E_{1,0,0}+E_{0,1,0}+E_{0,0,1}+E_{1,1,1},
 \]
where
\[
   E_{1,0,0}= \frac{3\lambda}{16} \sin^4\omega_1 +
   \frac{\lambda}{64} (2\sin\omega_1- A \cos\omega_1)^2 +
   \frac{1-16\lambda}{64} \sin^2\omega_1 (1+\lambda A^2) ;
\]
the other $E_{i,j,k}$ are given by the same coordinate transformations
as $p_{i,j,k}$. Hence, for $ \frac{1}{16}-|p_{1,0,0}|^2-|p_{0,1,0}|^2-|p_{0,0,1}|^2-|p_{1,1,1}|^2$, we obtain an sos
decomposition  with $12$ summands $g_j$, each of which has only one nonzero isotypical component.
\end{itemize}

Thus, for the trivariate interpolatory subdivision
scheme with tension parameter $0\le\lambda<1/16$, by Theorem~\ref{th:LaiSt},
we have explicitly constructed a tight frame with 41 generators $q_j$
as in Theorem \ref{th:UEP}.

\item[c)]
For $\lambda=1/16$, the sum rules of order $4$
are satisfied.
 In this particular case,
 the scheme is  $C^1$  and the Hessian of $f$ at $\bone$ is the zero-matrix,
 thus the result
 of Theorem~\ref{th:existence-ddim} is not applicable. Nevertheless,  the sos decomposition
 of
 $1-\sum p^{\sigma *} p^\sigma$ in b), with further simplifications
 for $\lambda=1/16$,
gives a tight frame with 31 generators
for the trivariate interpolatory subdivision
scheme.
\end{itemize}
\end{Example}

%%%%%%%%%%%%%%%%%%%%%%%%%%%%%%%%%%%%%%%%%%%%%%%%%%%%%%%%%%%%%%%%%%%%%%%%%%%%%
\subsection{Constructions of tight wavelet frames} \label{subsec:construction}
%%%%%%%%%%%%%%%%%%%%%%%%%%%%%%%%%%%%%%%%%%%%%%%%%%%%%%%%%%%%%%%%%%%%%%%%%%%%%

Lemma \ref{lem:subQMFiso} part $(a)$ sometimes yields an  elegant
method for determining the sum of squares decomposition of the polynomial $f$ in \eqref{def:f} and, thus,
constructing the trigonometric polynomials $q_j$ in Theorem \ref{th:UEP}. Note that
\begin{equation} \label{idea:method}
 f\>=1-\sum_{\sigma \in G} p^{\sigma *} p^\sigma \>=\>1-m\sum_{\chi\in G'} p_\chi^* p_\chi \>=\>m\sum_{\chi\in G'}
 \Bigl(\frac1{m^2}-p_\chi^* p_\chi \Bigr).
\end{equation}
So it suffices to find an sos decomposition for each of
the polynomials $m^{-2}-p_\chi^* p_\chi$, provided that they are all
nonnegative. This nonnegativity assumption is satisfied, for
example, for the special case when all coefficients $p(\alpha)$ of
$p$ are nonnegative. This is due to the simple fact that for
nonnegative $p(\alpha)$ we get
$$
 p^*_\chi p_\chi\>\le\>|p_\chi(\bone)|^2\>=\>m^{-2}
$$
on $\T^d$, for all $\chi\in G'$.

The last equality in \eqref{idea:method} allows us to simplify the construction of frame generators
considerably. In Example \ref{example:b111_2_construction} we apply
this method to the  three-directional piecewise linear box spline. Example \ref{ex:butterfly} illustrates the advantage of the representation
in \eqref{idea:method} for the butterfly scheme \cite{GDL}, an interpolatory subdivision method
with the corresponding mask $p \in \C[T]$ of a larger support, some of whose coefficients are negative. Example \ref{ex:jiang_oswald} shows that
our method is also applicable for at least one of the interpolatory $\sqrt{3}-$subdivision schemes studied in \cite{JO}.
For the three-dimensional
example that also demonstrates our constructive approach see Example \ref{ex:3D_butterfly} part $(b1)$.

%%%%%%%%%%%%%%%%%%%%%%%%%%%%%%5
\begin{Example}\label{example:b111_2_construction}%
Consider the three-directional piecewise linear box spline with the
symbol
$$p(z_1,z_2)\>=\>\frac18\,(1+z_1)(1+z_2)(1+z_1z_2),\quad z_j=
e^{-i\omega_j}.$$
The sos decomposition for the isotypical components yields
$$
 f\>=\>1-m\sum_{\chi\in G'} p_\chi^* p_\chi \>=\>\frac14\sin^2(\omega_1)+
 \frac14\sin^2(\omega_2)+\frac14\sin^2({\omega_1+\omega_2}).
$$
Thus, in \eqref{id:tildeH} we have a decomposition with $r=3$.
Since each of $h_1$, $h_2$, $h_3$ has only one isotypical
component, we get a representation $f=\tilde h_1^2+\tilde
h_2^2+\tilde h_3^2$ with $3$ $G$-invariant polynomials $\tilde
h_j$. By Theorem
\ref{th:LaiSt} we get $7$ frame generators. Note that the
method in \cite[Example 2.4]{LS06} yields $6$ generators of
slightly larger support. The method in \cite[Section 4]{CH} based on
properties of the Kronecker product leads to $7$ frame generators
whose support is the same as the one of $p$. One can also employ the
technique discussed in \cite[Section]{GR} and get $7$ frame
generators.
\end{Example}

Another prominent example of a subdivision scheme is the so-called
butterfly scheme. This example shows  the real advantage of
treating the isotypical components of $p$ separately for $p$ with larger support.

%%%%%%%%%%%%%%%%%%%%%%%%%%5
\begin{Example} \label{ex:butterfly}
The butterfly scheme describes an interpolatory subdivision scheme
that generates a smooth regular surface
interpolating a given set of points \cite{GDL}. The trigonometric
polynomial $p$ associated with the butterfly scheme is given by
\begin{eqnarray*}
p(z_1,z_2) & = & \frac{1}{4}+\frac{1}{8}\Bigl(z_1+z_2+z_1z_2+z_1^{-1}+z_2^{-1}+
  z_1^{-1}z_2^{-1}\Bigr) \\
&& +\frac{1}{32}\Bigl(z_1^2z_2+z_1z_2^2+z_1z_2^{-1}+z_1^{-1}z_2+z_1^{-2}
  z_2^{-1}+z_1^{-1}z_2^{-2}\Bigr) \\
&& -\frac1{64}\Bigl(z_1^3z_2+z_1^3z_2^2+z_1^2z_2^3+z_1z_2^3+z_1^2
  z_2^{-1}+z_1z_2^{-2} \\
&& +z_1^{-1}z_2^2+z_1^{-2}z_2+z_1^{-3}z_2^{-1}+z_1^{-3}z_2^{-2}+
  z_1^{-2}z_2^{-3}+z_1^{-1}z_2^{-3}\Bigr).
\end{eqnarray*}
Its first isotypical component is $p_{0,0}=\frac14$, which is the case for every
interpolatory subdivision scheme. The other isotypical components, in terms of
$z_k=e^{-i\omega_k}$, $k=1,2$, are given by $p_{1,0}
(z_1,z_2)=\frac14\cos(\omega_1)+\frac1{16}\cos(\omega_1+2\omega_2)
-\frac1{32}\cos(3\omega_1+2\omega_2)-\frac1{32}\cos(\omega_1 -2\omega_2)$, i.e.,
$$
 p_{1,0}(z_1,z_2)\>=\>\frac14\cos(\omega_1)+\frac18\sin^2(\omega_1)
 \cos(\omega_1+2\omega_2),
$$
and $p_{0,1}(z_1,z_2)=p_{1,0}(z_2,z_1)$, $p_{1,1}(z_1,z_2)=p_{1,0}
(z_1z_2,z_2^{-1})$. Note that on $\T^2$
$$
 |p_\chi|\le\frac14 \quad  \hbox{for all}  \quad \chi\in G',
$$
thus, our method is applicable. Simple computation shows that
\begin{eqnarray*}
1-16\,\bigl|p_{1,0}(z_1,z_2)
\bigr|^2&=&1-\cos^2(\omega_1)-\cos(\omega_1)\sin^2(\omega_1)\cos
(\omega_1+2\omega_2)\\&-&\frac14\sin^4(\omega_1)\cos^2(\omega_1+ 2\omega_2).
\end{eqnarray*}
Setting $u_j:=\sin(\omega_j)$, $j=1,2$, $v:= \sin(\omega_1+\omega_2)$,
$v':=\sin(\omega_1-\omega_2)$, $w:=\sin (\omega_1+2\omega_2)$ and $w':=\sin
(2\omega_1+\omega_2)$, we get
$$1-16\,\bigl|p_{1,0}(z_1,z_2)\bigr|^2\>=\>\frac14\,u_1^2\Bigl(w^2+
(u_2^2+v^2)^2+2u_2^2+2v^2\Bigr).$$
Therefore,
\begin{eqnarray*}
f=1-\sum_{\sigma\in G} p^{\sigma *} p^\sigma & = & \frac14\Bigl(u_1^2u_2^2+u_1^2
  v^2+u_2^2v^2\Bigr)+\frac1{16}\Bigl(u_1^2w^2+u_2^2w'^2+v^2v'^2
  \Bigr) \\
&& +\frac1{16}\Bigl(u_1^2(u_2^2+v^2)^2+u_2^2(u_1^2+v^2)^2+v^2(u_1^2
  +u_2^2)^2\Bigr).
\end{eqnarray*}
This provides a decomposition $\displaystyle f=
\sum_{j=1}^9 h_j^* h_j$ into a sum of $9$ squares. As in the previous example, each
$h_j$ has only one nonzero isotypical component $h_{j,\chi_j}$. Thus, by part $(b)$
of Lemma \ref{lem:subQMFiso} and by Theorem \ref{th:LaiSt}, there
exists a tight frame with $13$ generators. Namely, as in the proof of Theorem \ref{th:LaiSt},
we get
\begin{eqnarray*}
 q_1(z_1,z_2)&=&\frac{1}{2}-\frac{1}{2}p(z_1,z_2), \quad q_2(z_1,z_2)=
 \frac{1}{2}z_1-2 p(z_1,z_2)p_{(1,0)}^*(z_1,z_2) \\
 q_3(z_1,z_2)&=&q_2(z_2,z_1), \quad q_4(z_1,z_2)=q_2(z_1z_2,z_2^{-1})\\
 q_{4+j}(z_1,z_2)&=&p(z_1,z_2) \widetilde{h}^*_{j,\chi_j}, \quad j=1, \dots, 9,
\end{eqnarray*}
where $\widetilde{h}_{j,\chi_j}$ are the lifted isotypical components defined as
in Lemma \ref{lem:subQMFiso}. Let $\cN=\{0, \dots ,7\}^2$, $p=\bp \cdot
\bx$ and $q_j=\bq_j \cdot \bx$ with $\bx= [
z^\alpha \ : \ \alpha \in \cN]^T$. The corresponding null-matrix $O \in \R^{64
\times 64}$ satisfying \eqref{id:null-matrices} is given by
$$
 \bx^* \cdot  O \cdot \bx= \bx^* \left[
 \sum_{j=1}^{13} \bq_j^T \bq_j -\hbox{diag}(\bp)+\bp^T \bp \right] \bx.
$$
Note that other factorizations of the positive semi-definite matrix
$\hbox{diag}(\bp)-\bp^T \bp+O$ of rank $13$ lead to other possible tight frames
with at least $13$ frame generators. An advantage of using semi-definite
programming techniques is that it can possibly yield $q_j$ of smaller degree and
reduce the rank of $\hbox{diag}(\bp)-\bp^T \bp+O$.

Using the technique of semi-definite programming the authors in \cite{CS08}
constructed numerically a tight frame for the butterfly scheme with 18 frame
generators. The advantage of our construction is that the frame
generators are determined analytically. The disadvantage is that their support
is approximately twice as large as that of the frame generators in  \cite{CS08}.
\end{Example}

The next example is one of the family of interpolatory $\sqrt{3}-$subdivision studied in \cite{JO}.
The associated dilation matrix is $\displaystyle{M=\left[\begin{array}{rr} 1&2\\-2&-1 \end{array} \right]}$
and $m=3$.

%%%%%%%%%%%%%%%%%%%%%%%%%%%%%%%%%%%%%
\begin{Example} \label{ex:jiang_oswald}
The symbol of the scheme is given by
$$
 p(z_1,z_2)= p_{(0,0)}(z_1,z_2)+ p_{(1,0)}(z_1,z_2)+ p_{(0,1)}(z_1,z_2)
$$
with isotypical components $p_{(0,0)}=\frac{1}{3}$,
\begin{eqnarray*}
   p_{(0,1)}(z_1,z_2)=\frac{4}{27}(z_2+z_1^{-1}+z_1z_2^{-1})-\frac{1}{27}(z_1^{-2}z_2^{2}+z_1^2+z_2^{-2})
\end{eqnarray*}
and $  p_{(1,0)}(z_1,z_2)=p_{(0,1)}(z_2,z_1)$. We have by Lemma
\ref{lem:subQMFiso} and due to the equality $|p_{(0,1)}(z_1,z_2)|^2=|p_{(1,0)}(z_1,z_2)|^2$
$$
 1-\sum_{\sigma \in G} p^{\sigma *} p^\sigma=2\left( \frac19-p_{(0,1)}^*p_{(0,1)}\right),
$$
thus it suffices to consider only
\begin{eqnarray*}
 \frac{1}{9}&-&|p_{(0,1)}(z_1,z_2)|^2=3^{-2}-27^{-2} \Big(51+16\cos(\omega_1+\omega_2)+16\cos(2\omega_1-\omega_2) \\
&+&16\cos(\omega_1-2\omega_2)+2\cos(2\omega_1+2\omega_2)+2\cos(2\omega_1-4\omega_2)\\&+&2\cos(4\omega_1-2\omega_2)
-8\cos(3\omega_1)-8\cos(3\omega_2)-8\cos(3\omega_1-3\omega_2\Big).
\end{eqnarray*}
Numerical tests show that this polynomial is nonnegative.
\end{Example}

%%%%%%%%%%%%%%%%%%%%%%%%%%%%%%%%%%%%%%%%%%%%%%%%%%%%%%%

\end{document}